\documentclass[11pt]{amsart}
\usepackage{eurosym}
\usepackage{amssymb}
\usepackage{amsfonts}
\usepackage{amsmath}
\usepackage{xcolor}
\usepackage{bbm}

\newtheorem{theorem}{Theorem}[section]
\newtheorem{proposition}{Proposition}[section]

\theoremstyle{definition}

\newtheorem{example}[theorem]{Example}

\theoremstyle{remark}
\newtheorem{remark}[theorem]{Remark}
\numberwithin{equation}{section}

\subjclass{}
\keywords{Quasilinear singular system, p-Laplacian, multiplicity, nodal solution, sub-supersolution, Leray–Schauder topological degree}

\begin{document}
\title[Multiple solutions for singular quasilinear elliptic systems]{%
Constant sign and nodal solutions for singular quasilinear elliptic systems}
\subjclass[2020]{35J62, 35J92, 35B09, 35B99}
\keywords{Quasilinear singular system, p-Laplacian, multiplicity, nodal
solution, sub-supersolution, Leray-Schauder topological degree}

\begin{abstract}
We consider singular quasilinear elliptic systems with homogeneous Dirichlet
boundary condition. Using Leray-Schauder topological degree, combined with
the sub-supersolutions method and suitable truncation arguments, we
establish the existence of at least three nontrivial solutions, two of which
are of opposite constant sign. The third solution is nodal and exhibits
components of at least opposite constant sign. In the case of a sign-coupled
system, these components are of changing and synchronized sign.
\end{abstract}

\author{Nouredine Medjoudj\textsuperscript{$\star$}}
\address{\textsuperscript{$\star$} Applied Mathematics Laboratory, Faculty
of Exact Sciences,\\
A. Mira Bejaia University, Targa Ouzemour, 06000 Bejaia, Algeria}
\author{Abdelkrim Moussaoui\textsuperscript{$\star, \dagger$}}
\address{$\dagger$ Department of Physico-Chemical Biology, Faculty of
Natural and Life Sciences,\\
A. Mira Bejaia University, Algeria}
\email{nouredine.medjoudj@univ-bejaia.dz}
\email{abdelkrim.moussaoui@univ-bejaia.dz}
\maketitle

\section{Introduction}

\label{S1}

Let $\Omega $ is a bounded domain in $%
\mathbb{R}
^{N}$ $\left( N\geq 2\right) $ with a smooth boundary $\partial \Omega $. We
consider the following system of quasilinear elliptic equations%
\begin{equation*}
(\mathrm{P})\qquad \left\{ 
\begin{array}{ll}
-\Delta _{p_{1}}u_{1}=f_{1}(x,u_{1},u_{2}) & \text{in }\Omega \\ 
-\Delta _{p_{2}}u_{2}=f_{2}(x,u_{1},u_{2}) & \text{in }\Omega \\ 
u_{1},u_{2}=0\text{ } & \text{on }\partial \Omega ,%
\end{array}%
\right.
\end{equation*}%
where $1<p_{i}<\infty $ and $\Delta _{p_{i}}$ stands for the $p_{i}$%
-Laplacian differential operator defined by $-\Delta
_{p_{i}}u_{i}=-div(|\nabla u_{i}|^{p_{i}-2}\nabla u_{i}),$ for $u_{i}\in
W_{0}^{1,p_{i}}(\Omega ).$ By a solution of problem $(\mathrm{P})$ we mean $%
(u_{1},u_{2})\in W_{0}^{1,p_{1}}\left( \Omega \right) \times
W_{0}^{1,p_{2}}\left( \Omega \right) $ such that 
\begin{equation*}
\left\{ 
\begin{array}{ll}
\int_{\Omega }\left\vert \nabla u_{1}\right\vert ^{p_{1}-2}\nabla
u_{1}\nabla \varphi _{1}\ \mathrm{d}x & =\int_{\Omega
}f_{1}(x,u_{1},u_{2})\varphi _{1}\ \mathrm{d}x, \\ 
\int_{\Omega }\left\vert \nabla u_{2}\right\vert ^{p_{2}-2}\nabla
u_{2}\nabla \varphi _{2}\ \mathrm{d}x & =\int_{\Omega
}f_{2}(x,u_{1},u_{2})\varphi _{2}\ \mathrm{d}x,%
\end{array}%
\right.
\end{equation*}%
for all $\varphi _{i}\in W_{0}^{1,p_{i}}(\Omega ),\ i=1,2,$ provided the
integrals in the right-hand side of the above identities exist. The
nonlinear reaction terms $f_{1}(x,u_{1},u_{2})$ and $f_{2}(x,u_{1},u_{2})$
are Carath\'{e}odory functions that exhibit a singularity that, without loss
of generality, is located at zero. This occurs under the following
hypothesis on the nonlinearities $f_{i}(x,s_{1},s_{2})$ ($i=1,2$) reflecting
the singular character of problem $(\mathrm{P})$:

\begin{description}
\item[$(\mathrm{H}_{1})$] 
\begin{equation*}
\begin{array}{l}
\lim_{s_{i}\rightarrow 0^{\pm }}f_{i}(x,s_{1},s_{2})=+\infty ,\text{ \
uniformly for a.e. }x\in \Omega \text{, for all }s_{j}>0, \\ 
\lim_{s_{i}\rightarrow 0^{\pm }}f_{i}(x,s_{1},s_{2})=-\infty ,\text{ \
uniformly for a.e. }x\in \Omega \text{, for all }s_{j}<0,%
\end{array}%
\end{equation*}%
\ for $i,j=1,2,i\neq j.$
\end{description}

Singularities are involved in a wide range of important elliptic problems
which have been subject to intensive study in recent years, see for instance 
\cite{AM, AMT, DM, DM1, DM2, 3, GHM, GST2, KM, MMM, MM3, MM2, MM1, M, M2,
M3, M4, MKT, MV2} and there references. Their impact on the structure of the
problem prevents the direct application of conventional nonlinear analysis
techniques, including sub-supersolutions and fixed point methods. Moreover,
singularities make difficult any study attendant to sign properties of
solutions, particularly those that change sign, as they inevitably pass
through zero. These type of solutions adhere to the concept of nodal
solutions, which are defined as neither positive nor negative. They are
necessarily sign changing functions in the scalar case while in the context
of systems, they incorporate several types of solutions depending on the
sign of their components. By this definition, a solution of a system whose
components at least are not of the same constant sign is nodal. As a
reminder, a solution is said to be positive (resp. negative) if its
components are positive (resp. negative).

From our current knowledge, the question of nodal solutions for singular
systems has been rarely investigated in the literature. Actually, \cite{M4}
established the existence of a nodal solution for singular Gierer-Meinhardt
type system whose components are sign-changing and synchronized. In \cite{M}%
, a solution exhibiting analogous characteristics is investigated for a
class of singular semilinear elliptic systems while, the quasilinear case
was examined for singular Lane-Emden type systems in \cite{M3}. The latter
establishes that a sign-changing and synchronized nodal solution exists,
provided that the system is sign-coupled and subject to a strong singularity
effect. We point out that the aforementioned issues are readily noticeable
for the system $(\mathrm{P})$ which, under assumption $(\mathrm{H}_{1})$,
exhibits singularities near the origin. Additionally, the absence of a
variational structure further complicates its examination because
variational techniques are inapplicable. (see the growth conditions $(%
\mathrm{H}_{2})$-$(\mathrm{H}_{4})$ below).

Our main interest in this work is to establish multiplicity result for
singular system $(\mathrm{P})$ with a precise sign information. In this
respect, we first establish the existence of opposite constant-sign
solutions to system $(\mathrm{P})$. Specifically, in addition to $(\mathrm{H}%
_{1})$, we assume:

\begin{description}
\item[$(\mathrm{H}_{2})$] \textit{There exist constants }$M_{i}>0,$\textit{\
and }$\alpha _{i},\beta _{i}\in 
\mathbb{R}
\backslash \{0\}$\textit{\ such that} 
\begin{equation*}
-1<\alpha _{1},\beta _{2}<0\text{, \ }|\beta _{1}|<\min \left\{
1,p_{1}-1\right\} ,\text{ \ \ }|\alpha _{2}|<\min \left\{ 1,p_{2}-1\right\} ,
\end{equation*}
\begin{equation*}
f_{i}(x,s,t)\leq M_{i}(1+s^{\alpha _{i}}+t^{\beta _{i}})\text{ \ and \ }
f_{i}(x,-s,-t)\geq -M_{i}(1+s^{\alpha _{i}}+t^{\beta _{i}}),
\end{equation*}
\ for a.e\textit{. }$x\in \Omega ,$ all $s,t>0,$ for $i=1,2.$

\item[$(\mathrm{H}_{3})$] \textit{There exist constants }$m_{i}>0$\textit{\
and }$\hat{\alpha}_{i},\hat{\beta}_{i}\in 
\mathbb{R}
\backslash \{0\}$\textit{\ such that} 
\begin{equation*}
\alpha _{i}\geq \hat{\alpha}_{i},\ \beta _{i}\geq \hat{\beta}_{i}\;\text{and 
}\ \hat{\alpha}_{i}+\hat{\beta}_{i}>-\min \left\{ 1,p_{i}-1\right\} ,
\end{equation*}
\begin{equation*}
f_{i}(x,s,t)\geq m_{i}(s^{\hat{\alpha}_{i}}+t^{\hat{\beta}_{i}})\text{ \ and
\ }f_{i}(x,-s,-t)\leq -m_{i}(s^{\hat{\alpha}_{i}}+t^{\hat{\beta}_{i}}),
\end{equation*}
for a.e\textit{. }$x\in \Omega ,$ all $s,t>0,$ for $i=1,2.$
\end{description}

These assumptions suffice to show the existence of a positive solution $%
(u_{1,+},u_{2,+})$ and negative solution $(u_{1,-},u_{2,-})$ to system $(%
\mathrm{P})$. It is stated in Theorem \ref{T1}.\ The proof is achieved
through the sub-supersolutions theorem for singular systems in \cite[Theorem
2]{KM}. It was established that the aforementioned constant-sign solutions
are located in positive and negative rectangles formed by two opposite
constant sign sub-supersolutions pairs. The latter, which provide a
localization of these two solutions, are constructed by a choice of suitable
functions with an adjustment of adequate constants. It is of considerable
consequence to note that this choice is particularly relevant in regard to
the properties characterizing solutions of constant sign. Indeed, exploiting
the characteristic of the sub-supersolutions pairs, it is established that
every constant-sign solution enclosed within the rectangle formed by the
opposite supersolutions is inevitably outside the rectangle formed by the
opposite sign subsolutions. This is a fundamental point that sets the stage
for our second primary goal which represents a meaningful aspect of our
results. It is a matter of showing the existence of another nontrivial
solution to system $(\mathrm{P})$ that exhibits the additional property of
being nodal. We assume the following assumption:

\begin{description}
\item[$(\mathrm{H}_{4})$] \textit{There exist constants }$\eta _{1},\eta
_{2}>0$ \textit{such that} 
\begin{equation*}
|f_{i}(x,s,t)|\leq M_{i}(1+|s|^{\alpha _{i}}+|t|^{\beta _{i}}),
\end{equation*}%
\ for a.e\textit{. }$x\in \Omega ,$ for all $(s,t)\in \lbrack -\eta
_{1},\eta _{1}]\times \lbrack -\eta _{2},\eta _{2}],$ $s,t\neq 0,$ $i=1,2.$
\end{description}

The third solution to $(\mathrm{P})$ subjected to assumption $(\mathrm{H}%
_{4})$ is provided by Theorem \ref{T2}. The proof is based on topological
degree theory. It is applied to a regularized system $(\mathrm{P}%
^{\varepsilon })$, depending on parameter $\varepsilon >0$, whose study is
relevant for problem $(\mathrm{P})$. By virtue of a priori estimates derived
from suitable truncation depending on $\varepsilon >0$, two balls are
defined in Sobolev space $W_{0}^{1,p_{1}}(\Omega )\times
W_{0}^{1,p_{2}}(\Omega )$ such that one fits strictly into the other. The
larger ball encompasses all potential solutions, whereas the smaller ball
focuses on the area between the rectangles with opposite signs, defined by
the aforementioned opposite constant sign sub-supersolutions pairs. A
specific and crucial choice of homotopic functions corresponding to $(%
\mathrm{P}^{\varepsilon })$ ensures that on the large ball, the degree is $0$
while it is $1$ when the smaller ball is excluded. Then, the domain
additivity property of the Leray-Schauder degree implies that the degree on
the smallest ball is not zero, thereby showing that the latter contains at
least a nontrivial solution to problem $(\mathrm{P}^{\varepsilon })$. Then,
through a priori estimates and dominated convergence Theorem, we may pass to
the limit as $\varepsilon \rightarrow 0$ in $(\mathrm{P}_{\varepsilon })$.
This leads to a solution $(u_{1}^{\ast },u_{2}^{\ast })$ of $(\mathrm{P})$
which, according to its localization, does not coincide with the above
solutions $(u_{1,+},u_{2,+})$ and $(u_{1,-},u_{2,-})$. Thus, $(u_{1}^{\ast
},u_{2}^{\ast })$ is a third solution of $(\mathrm{P})$. In light of the
preceding property that characterizes constant sign solutions, which asserts
that such solutions are inevitably exterior to the smallest ball, $%
(u_{1}^{\ast },u_{2}^{\ast })$ is nodal. Namely, both components $%
u_{1}^{\ast }$ and $u_{2}^{\ast }$ are nontrivial and at least are not of
the same constant sign, neither positive, nor negative. It should be noted
that the process for returning to the original problem $(\mathrm{P})$,
resulting from the passage to the limit in $(\mathrm{P}_{\varepsilon })$ as $%
\varepsilon \rightarrow 0$, represents a significant part of the argument.
The control near the singularity of all the terms involved in the problem is
crucial and it requires a specific truncation.

Another important issue that is rarely addressed in the existing literature
is the question of nodal solutions $(u_{1}^{\ast },u_{2}^{\ast })$ for $(%
\mathrm{P})$ where both components $u_{1}^{\ast }$ and $u_{2}^{\ast }$ are
sign changing. Such solutions are studied for specific singular problems in 
\cite{M, M3, M4}. Our result, stated in Theorem \ref{T4}, shows that $%
u_{1}^{\ast }$ and $u_{2}^{\ast }$ change sign simultaneously under the
further assumption:

\begin{description}
\item[$(\mathrm{H}_{5})$] $f_{i}(x,s_{1},s_{2})\cdot sgn(s_{j})>0,$ for a.e. 
$x\in \Omega ,$ for all $s_{i},s_{j}\neq 0,$ for $i,j=1,2,i\neq j$.
\end{description}

Hypothesis $(\mathrm{H}_{5})$ reflects the fact that system $(\mathrm{P})$
is of sign-coupled. This means that the nonlinearities $f_{1}(x,u_{1},u_{2})$
and $f_{2}(x,u_{1},u_{2})$ must have the same sign as $u_{2}$ and $u_{1}$,
respectively. We emphasize that our result extends that of \cite{M3, M4},
established for Lane-Emden and Gierer-Meinhardt type problems, to a more
general quasilinear singular system $(\mathrm{P})$.

The rest of the paper is organized as follows. Section \ref{S2} deals with
the existence of constant-sign solutions for the system $(\mathrm{P}),$
while section \ref{S4} provides nodal solutions.

\section{Opposite constant sign solutions}

\label{S2}

In the sequel, the Banach spaces $W^{1,p}(\Omega )$ and $L^{p}(\Omega )$ are
equipped with the usual norms $\Vert \cdot \Vert _{1,p}$ and $\Vert \cdot
\Vert _{p}$, respectively, whereas the space $W_{0}^{1,p}(\Omega )$ is
endowed with the equivalent norm $\Vert \nabla u\Vert _{p}=(\int_{\Omega
}|\nabla u|^{p}\,\mathrm{d}x)^{\frac{1}{p}}$. We also utilize the H\"{o}lder
spaces $\mathcal{C}^{1}(\overline{\Omega })$ and $\mathcal{C}^{1,\tau }(%
\overline{\Omega }),$ $\tau \in (0,1)$. Hereafter, we denote by $d(x)$ the
distance from a point $x\in \overline{\Omega }$ to the boundary $\partial
\Omega $, where $\overline{\Omega }=\Omega \cup \partial \Omega $ is the
closure of $\Omega \subset 
\mathbb{R}
^{N}$. For $u,v\in \mathcal{C}^{1}(\overline{\Omega })$, the notation $u\ll
v $ means that $u(x)<v(x),\,\forall x\in \Omega ,$ and $\frac{\partial v}{%
\partial \eta }<\frac{\partial u}{\partial \eta }$ on $\partial \Omega .$

Let $\phi _{1,p_{i}}$ be the positive eigenfunction associated with the
principal eigenvalue $\lambda _{1,p_{i}}$ of $-\Delta _{p_{i}}$, that is,%
\begin{equation}
\begin{array}{c}
-\Delta _{p_{i}}\phi _{1,p_{i}}=\lambda _{1,p_{i}}|\phi
_{1,p_{i}}|^{p_{i}-2}\phi _{1,p_{i}}\text{ \ in }\Omega ,\text{ \ }\phi
_{1,p_{i}}=0\text{ \ on }\partial \Omega .%
\end{array}
\label{1}
\end{equation}%
We recall that there exists a constant $c_{0}>1$ such that%
\begin{equation}
c_{0}d(x)\geq \phi _{1,p_{i}}(x)\geq c_{0}^{-1}d(x),\text{ for all }x\in
\Omega .  \label{2}
\end{equation}%
Moreover, the eigenvalue $\lambda _{1,p_{i}}$ is characterized by 
\begin{equation}
\lambda _{1,p_{i}}=\inf_{u_{i}\in W_{0}^{1,p_{i}}(\Omega )\backslash \{0\}} 
\frac{\int_{\Omega }|\nabla u_{i}|^{p_{i}}\,\mathrm{d}x}{\int_{\Omega
}|u_{i}|^{p_{i}}\,\mathrm{d}x}.  \label{3}
\end{equation}

Inspired by \cite{DM}, the Dirichlet problems 
\begin{equation}
-\Delta _{p_{i}}y_{i}(x)=1+d(x)^{\alpha _{i}}+d(x)^{\beta _{i}}\text{ in }%
\Omega ,\text{ \ }y_{i}=0\text{ on }\partial \Omega ,  \label{4}
\end{equation}%
\begin{equation}
-\Delta _{p_{i}}z_{i}(x)=\left\{ 
\begin{array}{ll}
d(x)^{\hat{\alpha}_{i}}+d(x)^{\hat{\beta}_{i}} & \text{in \ }\Omega
\backslash \overline{\Omega }_{\delta }, \\ 
-1 & \text{in \ }\Omega _{\delta },%
\end{array}%
\right. ,\text{ \ }z_{i}=0\text{ \ on }\partial \Omega ,  \label{5}
\end{equation}%
have a unique positive solutions $y_{i},z_{i}\in \mathcal{C}^{1}\left( 
\overline{\Omega }\right) $ such that 
\begin{equation}
c^{-1}d(x)\leq z_{i}(x)\leq y_{i}(x)\leq cd(x)\text{ in }\Omega ,  \label{6}
\end{equation}%
where $c>1$ is a constant and $\Omega _{\delta }=\left\{ x\in \Omega
:d(x)<\delta \right\} ,$ with a fixed $\delta >0$ sufficiently small.

Fix a constant $C>1$ and set%
\begin{equation}
\underline{u}_{i}=C^{-1}z_{i}\text{ \ and \ }\overline{u}_{i}=Cy_{i},\text{ }%
i=1,2.  \label{7}
\end{equation}%
Due to (\ref{6}), this gives $\overline{u}_{i}\geq \underline{u}_{i}\ $ in $%
\ \overline{\Omega },$ so we can consider the ordered interval 
\begin{equation*}
\left[ \underline{u}_{i},\overline{u}_{i}\right] =\{v\in
W_{0}^{1,p_{i}}(\Omega ):\underline{u}_{i}(x)\leq v(x)\leq \overline{u}%
_{i}(x)\ \text{a.e.}\ x\in \Omega \}.
\end{equation*}%
Moreover, on account of (\ref{2}), (\ref{6}), (\ref{7}) and for $C>1$ large
enough, it holds%
\begin{equation}
\phi _{1,p_{i}}(x)\geq \underline{u}_{i}(x)\text{ for all }x\in \Omega .
\label{8}
\end{equation}%
Our first result deals with constant-sign solutions and is presented in the
following statement.

\begin{theorem}
\label{T1}Assume that $(\mathrm{H}_{2})$ and $(\mathrm{H}_{3})$ hold. Then,
problem $(\mathrm{P})$ admits two opposite constant-sign solutions $%
(u_{1,+},u_{2,+})$ and $(u_{1,-},u_{2,-})$ in $\mathcal{C}^{1,\tau }( 
\overline{\Omega })\times \mathcal{C}^{1,\tau }(\overline{\Omega }),$ for
certain $\tau \in (0,1).$ Moreover, every positive solution $%
(u_{1,+},u_{2,+})\in \mathcal{C}^{1}(\overline{\Omega })\times \mathcal{C}
^{1}(\overline{\Omega })$ and negative solution $(u_{1,-},u_{2,-})\in 
\mathcal{C}^{1}(\overline{\Omega })\times \mathcal{C}^{1}(\overline{\Omega }
) $ of $(\mathrm{P})$ within $[0,\overline{u}_{1}]\times \lbrack 0,\overline{
u}_{2}]$ and $[-\overline{u}_{1},0]\times \lbrack -\overline{u}_{2},0],$
respectively, satisfy 
\begin{equation}
u_{i,+}(x)\gg \underline{u}_{i}(x)\text{ \ and \ }u_{i,-}(x)\ll -\underline{%
u }_{i}(x),\text{ \ }\forall x\in \Omega .  \label{9}
\end{equation}
In particular, every positive solution $(u_{1,+},u_{2,+})\in
W_{0}^{1,p_{1}}(\Omega )\times W_{0}^{1,p_{2}}(\Omega )$ and negative
solution $(u_{1,-},u_{2,-})\in W_{0}^{1,p_{1}}(\Omega )\times
W_{0}^{1,p_{2}}(\Omega )$ of $(\mathrm{P})$ within $[0,\overline{u}
_{1}]\times \lbrack 0,\overline{u}_{2}]$ and $[-\overline{u}_{1},0]\times
\lbrack -\overline{u}_{2},0]$, respectively, satisfy 
\begin{equation}
u_{i,+}(x)\geq \underline{u}_{i}(x)\text{ \ and \ }u_{i,-}(x)\leq - 
\underline{u}_{i}(x),\text{ \ for a.e. }x\in \Omega .  \label{10}
\end{equation}
\end{theorem}

\begin{proof}
Let $(u_{1},u_{2})\in W_{0}^{1,p_{1}}(\Omega )\times W_{0}^{1,p_{2}}(\Omega
) $ within $[\underline{u}_{1},\overline{u}_{1}]\times \lbrack \underline{u}%
_{2},\overline{u}_{2}]$. Given that $\underline{u}_{i},\overline{u}_{i}\geq
0 $ in $\overline{\Omega }$, by (\ref{6}) and (\ref{7}) one has 
\begin{eqnarray*}
\overline{u}_{1}^{\alpha _{1}}+u_{2}^{\beta _{1}} &\leq &\left\{ 
\begin{array}{ll}
\overline{u}_{1}^{\alpha _{1}}+\overline{u}_{2}^{\beta _{1}} & \text{if }%
\beta _{1}>0 \\ 
\overline{u}_{1}^{\alpha _{1}}+\underline{u}_{2}^{\beta _{1}} & \text{if }%
\beta _{1}<0%
\end{array}%
\right. \leq \left\{ 
\begin{array}{ll}
(Cy_{1})^{\alpha _{1}}+(Cy_{2})^{\beta _{1}} & \text{if }\beta _{1}>0 \\ 
(Cy_{1})^{\alpha _{1}}+(C^{-1}z_{2})^{\beta _{1}} & \text{if }\beta _{1}<0%
\end{array}%
\right. \\
&\leq &\left\{ 
\begin{array}{ll}
(Cc^{-1}d(x))^{\alpha _{1}}+(Ccd(x))^{\beta _{1}} & \text{if }\beta _{1}>0
\\ 
(Cc^{-1}d(x))^{\alpha _{1}}+(C^{-1}c^{-1}d(x))^{\beta _{1}} & \text{if }%
\beta _{1}<0%
\end{array}%
\right. \\
&\leq &\left\{ 
\begin{array}{ll}
C^{\beta _{1}}\max \{c^{-\alpha _{1}},c^{\beta _{1}}\}(d(x)^{\alpha
_{1}}+d(x)^{\beta _{1}}) & \text{if }\beta _{1}>0 \\ 
C^{-\beta _{1}}\max \{c^{-\alpha _{1}},c^{-\beta _{1}}\}(d(x)^{\alpha
_{1}}+d(x)^{\beta _{1}}) & \text{if }\beta _{1}<0%
\end{array}%
\right. \\
&\leq &C^{|\beta _{1}|}c_{0}(d(x)^{\alpha _{1}}+d(x)^{\beta _{1}})\text{ in }%
\Omega ,
\end{eqnarray*}%
for certain positive constant $c_{0}:=c_{0}(c,\alpha _{1},\beta _{1})$.
Thus, using $(\mathrm{H}_{2})$, we deduce that 
\begin{eqnarray}
f_{1}(x,\overline{u}_{1},u_{2}) &\leq &M_{1}(1+\overline{u}_{1}^{\alpha
_{1}}+u_{2}^{\beta _{1}})  \label{11} \\
&\leq &M_{1}C^{|\beta _{1}|}(1+c_{0}(d(x)^{\alpha _{1}}+d(x)^{\beta _{1}}))%
\text{ in }\Omega .  \notag
\end{eqnarray}%
for $u_{2}\in \lbrack \underline{u}_{2},\overline{u}_{2}]$. Applying a
similar argument to the function $f_{2}(x,u_{1},\overline{u}_{2})$ we obtain 
\begin{eqnarray}
f_{2}(x,u_{1},\overline{u}_{2}) &\leq &M_{2}(1+u_{1}^{\alpha _{2}}+\overline{%
u}_{2}^{\beta _{2}}) \\
&\leq &M_{2}C^{|\alpha _{2}|}(1+c_{0}^{\prime }(d(x)^{\alpha
_{2}}+d(x)^{\beta _{2}}))\text{ in }\Omega ,  \notag  \label{12}
\end{eqnarray}%
for $u_{1}\in \lbrack \underline{u}_{1},\overline{u}_{1}],$ where $%
c_{0}^{\prime }>0$ is a constant. By (\ref{4}) and (\ref{7}) one has 
\begin{equation}
-\Delta _{p_{i}}\overline{u}_{i}=C^{p_{i}-1}(1+d(x)^{\alpha
_{i}}+d(x)^{\beta _{i}})\text{ \ in }\Omega ,\text{ }i=1,2.  \label{13}
\end{equation}%
Then, gathering (\ref{11})-(\ref{13}) together shows that $(\overline{u}_{1},%
\overline{u}_{2})$ is a supersolution for problem $(\mathrm{P})$, provided
that $C>1$ is large enough. From (\ref{6}) and (\ref{7}), we have 
\begin{eqnarray*}
\underline{u}_{1}^{\hat{\alpha}_{1}}+u_{2}^{\hat{\beta}_{1}} &\geq &\left\{ 
\begin{array}{ll}
\underline{u}_{1}^{\hat{\alpha}_{1}}+\underline{u}_{2}^{\hat{\beta}_{1}} & 
\text{if }\hat{\beta}_{1}>0 \\ 
\underline{u}_{1}^{\hat{\alpha}_{1}}+\overline{u}_{2}^{\hat{\beta}_{1}} & 
\text{if }\hat{\beta}_{1}<0%
\end{array}%
\right. \geq \left\{ 
\begin{array}{ll}
(C^{-1}z_{1})^{\hat{\alpha}_{1}}+(C^{-1}z_{2})^{\hat{\beta}_{1}} & \text{if }%
\hat{\beta}_{1}>0 \\ 
(C^{-1}z_{1})^{\hat{\alpha}_{1}}+(Cy_{2})^{\hat{\beta}_{1}} & \text{if }\hat{%
\beta}_{1}<0%
\end{array}%
\right. \\
&\geq &\left\{ 
\begin{array}{ll}
(C^{-1}cd(x))^{\hat{\alpha}_{1}}+(C^{-1}c^{-1}d(x))^{\hat{\beta}_{1}} & 
\text{if }\hat{\beta}_{1}>0 \\ 
(C^{-1}cd(x))^{\hat{\alpha}_{1}}+(Ccd(x))^{\hat{\beta}_{1}} & \text{if }\hat{%
\beta}_{1}<0%
\end{array}%
\right. \\
&\geq &\left\{ 
\begin{array}{ll}
C^{-\hat{\beta}_{1}}\min \{c^{\hat{\alpha}_{1}},c^{-\hat{\beta}_{1}}\}\text{ 
}(d(x)^{\hat{\alpha}_{1}}+d(x)^{\hat{\beta}_{1}}) & \text{if }\hat{\beta}%
_{1}>0 \\ 
C^{\hat{\beta}_{1}}\min \{c^{\hat{\alpha}_{1}},c^{\hat{\beta}_{1}}\}\text{ }%
(d(x)^{\hat{\alpha}_{1}}+d(x)^{\hat{\beta}_{1}}) & \text{if }\hat{\beta}%
_{1}<0%
\end{array}%
\right. \\
&\geq &C^{-|\beta _{1}|}c_{1}^{\prime }(d(x)^{\hat{\alpha}_{1}}+d(x)^{\hat{%
\beta}_{1}})\text{ in }\Omega ,
\end{eqnarray*}%
for a positive constant $c_{1}^{\prime }:=c_{1}(c,\hat{\alpha}_{1},\hat{\beta%
}_{1})$. On account of $(\mathrm{H}_{3})$ we infer that 
\begin{equation}
f_{1}(x,\underline{u}_{1},u_{2})\geq m_{1}(\underline{u}_{1}^{\hat{\alpha}%
_{1}}+u_{2}^{\hat{\beta}_{1}})\geq C^{-|\hat{\beta}_{1}|}c_{1}^{\prime
}(d(x)^{\hat{\alpha}_{1}}+d(x)^{\hat{\beta}_{1}})\text{ in }\Omega .
\label{14}
\end{equation}%
for $u_{2}\in \lbrack \underline{u}_{2},\overline{u}_{2}]$. A similar
argument applied to the function $f_{2}(x,u_{1},\overline{u}_{2})$ gives 
\begin{equation}
f_{2}(x,u_{1},\underline{u}_{2})\geq m_{2}(\underline{u}_{1}^{\hat{\alpha}%
_{2}}+u_{2}^{\hat{\beta}_{2}})\geq C^{-|\hat{\alpha}_{2}|}c_{0}^{\prime
}(d(x)^{\hat{\alpha}_{2}}+d(x)^{\hat{\beta}_{2}})\text{ in }\Omega .
\label{15}
\end{equation}%
for $u_{1}\in \lbrack \underline{u}_{1},\overline{u}_{1}],$ where $%
c_{1}^{\prime }>0$ is a constant depending on $c,\hat{\alpha}_{2}$ and $\hat{%
\beta}_{2}$. Then, bearing in mind from (\ref{4}) and (\ref{7}) that 
\begin{equation}
-\Delta _{p_{i}}\underline{u}_{i}=C^{-(p_{i}-1)}(d(x)^{\hat{\alpha}%
_{i}}+d(x)^{\hat{\beta}_{i}})\text{ \ in }\Omega ,\text{ for }i=1,2,
\label{16}
\end{equation}%
gathering (\ref{14})-(\ref{16}) together, for $C>1$ is large enough, we
conclude that $(\underline{u}_{1},\underline{u}_{2})$ is a subsolution for
problem $(\mathrm{P})$.

Consequently, thanks to \cite[Theorem 2]{KM}, there exists a positive
solution $(u_{1,+},u_{2,+})\in \mathcal{C}^{1,\tau }(\overline{\Omega }
)\times \mathcal{C}^{1,\tau }(\overline{\Omega }),$ $\tau \in (0,1),$ for
problem $(\mathrm{P})$ within $[\underline{u}_{1},\overline{u}_{1}]\times
\lbrack \underline{u}_{2},\overline{u}_{2}]$. Furthermore, by symmetry,
exploiting the growth conditions $(\mathrm{H}_{2})$ and $(\mathrm{H}_{3})$
when $s,t<0$, the existence of a negative solution $(u_{1,-},u_{2,-})\in 
\mathcal{C}^{1,\tau }(\overline{\Omega })\times \mathcal{C}^{1,\tau }( 
\overline{\Omega })$ for system $(\mathrm{P}),$ located in $[-\overline{u}
_{1},-\underline{u}_{1}]\times \lbrack -\overline{u}_{2},-\underline{u}
_{2}], $ is shown through a quite similar argument.

We are going to show the first inequality (\ref{9}), the second one follows
similarly. With this aim, let $(u_{1,+},u_{2,+})\in \lbrack 0,\overline{u}%
_{1}]\times ([0,\overline{u}_{2}]$ be a positive solution of $(\mathrm{P})$.
For $-(p_{1}-1)<\hat{\alpha}_{1}<0<\hat{\beta}_{1}$, observe that 
\begin{equation*}
\begin{array}{l}
C^{-(p_{1}-1)}(d(x)^{\hat{\alpha}_{1}}+d(x)^{\hat{\beta}%
_{1}})<C^{-(p_{1}-1)}(\delta ^{\hat{\alpha}_{1}}+d(x)^{\hat{\beta}_{1}}) \\ 
<(cC)^{\hat{\alpha}_{1}}\text{ }d(x)^{\hat{\alpha}_{1}}\text{ \ in }\Omega
\backslash \overline{\Omega }_{\delta },%
\end{array}%
\end{equation*}%
while, for $-(p_{1}-1)<\hat{\alpha}_{1},\hat{\beta}_{1}<0,$ we get 
\begin{equation*}
\begin{array}{l}
C^{-(p_{1}-1)}(d(x)^{\hat{\alpha}_{1}}+d(x)^{\hat{\beta}%
_{1}})<C^{-(p_{1}-1)}(\delta ^{\hat{\alpha}_{1}}+\delta ^{\hat{\beta}_{1}})
\\ 
<\min \{(cC)^{\hat{\alpha}_{1}},(cC)^{\hat{\beta}_{1}}\}\text{ }(d(x)^{\hat{%
\alpha}_{1}}+d(x)^{\hat{\beta}_{1}}),\text{ in }\Omega \backslash \overline{%
\Omega }_{\delta },%
\end{array}%
\end{equation*}%
provided $C>1$ large enough. From $(\mathrm{H}_{3})$ one has 
\begin{equation*}
f_{i}(x,u_{1,+},u_{2,+})\geq m_{i}(u_{1,+}^{\hat{\alpha}_{i}}+u_{2,+}^{\hat{%
\beta}_{i}})\text{ in }\Omega ,
\end{equation*}%
while, by (\ref{6}) and (\ref{7}), it holds 
\begin{eqnarray*}
u_{1,+}^{\hat{\alpha}_{1}}+u_{2,+}^{\hat{\beta}_{1}} &\geq &\left\{ 
\begin{array}{ll}
\overline{u}_{1}^{\hat{\alpha}_{1}} & \text{if }\hat{\beta}_{1}>0 \\ 
\overline{u}_{1}^{\hat{\alpha}_{1}}+\overline{u}_{2}^{\hat{\beta}_{1}} & 
\text{if }\hat{\beta}_{1}<0%
\end{array}%
\right. =\left\{ 
\begin{array}{ll}
(Cy_{1})^{\hat{\alpha}_{1}} & \text{if }\hat{\beta}_{1}>0 \\ 
(Cy_{1})^{\hat{\alpha}_{1}}+(Cy_{2})^{\hat{\beta}_{1}} & \text{if }\hat{\beta%
}_{1}<0%
\end{array}%
\right. \\
&\geq &\left\{ 
\begin{array}{ll}
(Ccd(x))^{\hat{\alpha}_{1}} & \text{if }\hat{\beta}_{1}>0 \\ 
(Ccd(x))^{\hat{\alpha}_{1}}+(Ccd(x))^{\hat{\beta}_{1}} & \text{if }\hat{\beta%
}_{1}<0%
\end{array}%
\right. \\
&\geq &\left\{ 
\begin{array}{ll}
(cC)^{\hat{\alpha}_{1}}\text{ }d(x)^{\hat{\alpha}_{1}} & \text{if }\hat{\beta%
}_{1}>0 \\ 
\min \{(cC)^{\hat{\alpha}_{1}},(cC)^{\hat{\beta}_{1}}\}\text{ }(d(x)^{\hat{%
\alpha}_{1}}+d(x)^{\hat{\beta}_{1}}) & \text{if }\hat{\beta}_{1}<0%
\end{array}%
\right. ,\text{ \ in }\Omega .
\end{eqnarray*}%
Set the function $\underline{\mathrm{X}}_{i}:\Omega \rightarrow \mathbb{R}$
given by 
\begin{equation*}
\underline{\mathrm{X}}_{i}(x):=C^{-(p_{i}-1)}\left\{ 
\begin{array}{ll}
d(x)^{\hat{\alpha}_{i}}+d(x)^{\hat{\beta}_{i}} & \text{in \ }\Omega
\backslash \overline{\Omega }_{\delta }, \\ 
-1 & \text{in \ }\Omega _{\delta },%
\end{array}%
\right.
\end{equation*}%
and note that $\underline{\mathrm{X}}_{i}\in L^{\infty }(\Omega )$ ($i=1,2$
). In view of (\ref{15}) and (\ref{16}), for each compact set $\mathrm{K}%
\subset \subset \Omega $, there is a constant $\tau =\tau (\mathrm{K})>0$
such that $\underline{\mathrm{X}}_{i}(x)+\tau <f{_{i}}({x,}u_{1,+},u_{2,+})$
a.e. in $\Omega \cap \mathrm{K},$ provided $C>1$ is sufficiently large. By
the strong comparison principle \cite[Proposition 2.6]{AR}, we infer that
property (\ref{9}) holds true. A quite similar argument shows the
corresponding property for negative solution. Repeating the same argument
and applying the weak comparison principle \cite[Proposition 2.6]{AR}, we
show the property (\ref{10}). This completes the proof.
\end{proof}

\begin{remark}
A careful inspection of the proof of Theorem \ref{T1} shows that the
constant $C>1$ in (\ref{7}) can be precisely estimated.
\end{remark}

\section{Nodal solutions}

\label{S4}

This section focuses on nodal solutions for problem $(\mathrm{P})$. In view
of Theorem \ref{T1}, both opposite constant-sign solutions $%
(u_{1,+},u_{2,+}) $ and $(u_{1,-},u_{2,-})$ of $(\mathrm{P})$ are in $]%
\underline{u}_{1},\overline{u}_{1}[\times ]\underline{u}_{2},\overline{u}%
_{2}[$ and $]-\overline{u}_{1},-\underline{u}_{1}[\times ]-\overline{u}_{2},-%
\underline{u}_{2}[$, respectively. Thus, any solution for problem $(\mathrm{P%
})$ in $[-\underline{u}_{1},\underline{u}_{1}]\times \lbrack -\underline{u}%
_{2},\underline{u}_{2}]$ is necessarily nodal. The main result is stated as
follows.

\begin{theorem}
\label{T2}Assume $(\mathrm{H}_{1})-(\mathrm{H}_{4})$ hold such that 
\begin{equation}
\alpha _{2}\geq \hat{\alpha}_{2}>0\text{ \ and }\ \beta _{1}\geq \hat{\beta}%
_{1}>0.  \label{17}
\end{equation}%
Then, system $(\mathrm{P})$ possesses nodal solutions $(u_{1}^{\ast
},u_{2}^{\ast })$ in $W_{0}^{1,p_{1}}(\Omega )\times W_{0}^{1,p_{2}}(\Omega
) $ where components $u_{1}^{\ast }$ and $u_{2}^{\ast }$ are nontrivial and
at least are not of the same constant sign.
\end{theorem}

\begin{remark}
\label{R1}Solutions $(u_{1},u_{2})\in W_{0}^{1,p_{1}}(\Omega )\times
W_{0}^{1,p_{2}}(\Omega )$ of $(\mathrm{P})$ satisfy $u_{1}(x),u_{2}(x)\neq 0$%
\ for a.e.\ $x\in \Omega .$ Otherwise, if $u_{1}$ and $u_{2}$ vanish on a
set of positive measure, then $(\mathrm{P})$ and hypothesis $(\mathrm{H}%
_{1}) $ are contradictory.
\end{remark}

\subsection{The regularized system}

For all $\varepsilon \in (0,1),$ we state the auxiliary system%
\begin{equation*}
(\mathrm{P}^{\varepsilon })\qquad \left\{ 
\begin{array}{ll}
-\Delta _{p_{1}}u_{1}=f_{1}(x,u_{1}+\gamma _{\varepsilon
}(u_{1}),u_{2}+\gamma _{\varepsilon }(u_{2})) & \text{in }\Omega \\ 
-\Delta _{p_{2}}u_{2}=f_{2}(x,u_{1}+\gamma _{\varepsilon
}(u_{1}),u_{2}+\gamma _{\varepsilon }(u_{2})) & \text{in }\Omega \\ 
u_{1},u_{2}=0\text{ } & \text{on }\partial \Omega ,%
\end{array}
\right.
\end{equation*}%
where%
\begin{equation}
\gamma _{\varepsilon }(s)=\varepsilon \text{ }(\frac{1}{2}+sgn(s)),\;\forall
s\in \mathbb{\ 
\mathbb{R}
}.  \label{18}
\end{equation}%
Our goal is to prove that $(\mathrm{P}^{\varepsilon })$ admits a solution $%
(u_{1,\varepsilon },u_{2,\varepsilon })$ within $[-\underline{u}_{1},%
\underline{u}_{1}]\times \lbrack -\underline{u}_{2},\underline{u}_{2}]$ and
then, passing to the limit as $\varepsilon \rightarrow 0$, to get the
existence of the desired solution $(u_{1}^{\ast },u_{2}^{\ast })$ for
problem $(\mathrm{P})$. The existence result regarding the regularized
system $(\mathrm{P}^{\varepsilon })$ is stated as follows.

\begin{theorem}
\label{T3} Assume that $(\mathrm{H}_{1})-(\mathrm{H}_{4})$ and (\ref{17})
hold. Then, the system $(\mathrm{P}^{\varepsilon })$ possesses solutions $%
(u_{1,\varepsilon },u_{2,\varepsilon })\in \mathcal{C}^{1,\tau }(\overline{%
\Omega })\times \mathcal{C}^{1,\tau }(\overline{\Omega })$ for some $\tau
\in (0,1)$ within $\left[ -\underline{u}_{1},\underline{u}_{1}\right] \times %
\left[ -\underline{u}_{2},\underline{u}_{2}\right] .$
\end{theorem}

For any $R>0$, set%
\begin{equation*}
\mathcal{B}_{\underline{u}}(0)=\left\{ (u_{1},u_{2})\in \mathcal{B}%
_{R}(0)\,:-\underline{u}_{i}\leq u_{i}\leq \underline{u}_{i},\text{ }%
i=1,2\right\} ,
\end{equation*}%
where $\mathcal{B}_{R}(0)$ denotes the ball in $L^{p_{1}}(\Omega )\times
L^{p_{2}}(\Omega )$ centered at $0$ of radius $R>0$.

\subsubsection{\textbf{The degree on }$\mathcal{B}_{R_{\protect\varepsilon %
}}(0)$}

Bearing in mind (\ref{18}), we introduce the truncation 
\begin{equation}
\mathcal{T}_{i,\varepsilon }(u_{i}(x))=\gamma _{\varepsilon
}(u_{i}(x))+\left\{ 
\begin{array}{ll}
\overline{u}_{i}(x) & \text{if \ }u_{i}(x)\geq \overline{u}_{i}(x) \\ 
u_{i}(x) & \text{if \ }-\overline{u}(x)\leq u_{i}(x)\leq \overline{u}_{i}(x)
\\ 
-\overline{u}_{i}(x) & \text{if \ }u_{i}(x)\leq -\overline{u}_{i}(x)%
\end{array}%
\right. ,  \label{19}
\end{equation}%
\ for a.a. $x\in \overline{\Omega },$ for $i=1,2,$ all $\varepsilon \geq 0.$
From (\ref{18}) and (\ref{7}), we derive that%
\begin{equation}
\frac{\varepsilon }{2}\leq |\mathcal{T}_{i,\varepsilon }(u_{i})|\leq \frac{%
3\varepsilon }{2}+C||y_{i}||_{\infty },\text{ for }i=1,2.  \label{20}
\end{equation}%
In particular, for $\varepsilon =0$ in (\ref{19}), we have%
\begin{equation}
\mathcal{T}_{i,0}(u_{i}(x))=\left\{ 
\begin{array}{ll}
\overline{u}_{i}(x) & \text{if \ }u_{i}(x)\geq \overline{u}_{i}(x) \\ 
u_{i}(x) & \text{if \ }-\overline{u}(x)\leq u_{i}(x)\leq \overline{u}_{i}(x)
\\ 
-\overline{u}_{i}(x) & \text{if \ }u_{i}(x)\leq -\overline{u}_{i}(x)%
\end{array}%
\right.
\end{equation}%
and%
\begin{equation}
|\mathcal{T}_{i,0}(u_{i})|\leq C||y_{i}||_{\infty },\text{ for }i=1,2.
\end{equation}%
We shall study the homotopy class of problem%
\begin{equation*}
(\mathrm{P}_{t}^{\varepsilon })\qquad \left\{ 
\begin{array}{l}
-\Delta _{p_{i}}u_{i}=\mathrm{F}_{i,t}^{\varepsilon }({x,}u_{1},u_{2})\text{
in }\Omega , \\ 
u_{i}=0\text{ \ on }\partial \Omega ,\text{ \ }i=1,2,%
\end{array}%
\right.
\end{equation*}%
with%
\begin{equation*}
\mathrm{F}_{i,t}^{\varepsilon }(x,u_{1},u_{2}):=tf_{i}(x,\mathcal{T}%
_{1,\varepsilon }(u_{1}),\mathcal{T}_{2,\varepsilon }(u_{2}))+(1-t)(\lambda
_{1,p_{i}}(\mathcal{T}_{i,0}(u_{i}^{+}))^{p_{i}-1}+1),
\end{equation*}%
for $\varepsilon \in (0,1),$ for $t\in \lbrack 0,1],$ where $s^{+}:=\max
\{0,s\}$ and $s^{-}:=\max \{0,-s\},$ for all $s\in 
\mathbb{R}
$. Note that in view of $(\mathrm{H}_{1}),$ any solution $(u_{1,\varepsilon
},u_{2,\varepsilon })\in W_{0}^{1,p_{1}}(\Omega )\times
W_{0}^{1,p_{2}}(\Omega )$ of $(\mathrm{P}_{t}^{\varepsilon })$ satisfies $%
u_{1,\varepsilon }(x),u_{2,\varepsilon }(x)\neq 0$ for a.e. $x\in \Omega $.
Hence, $\mathrm{F}_{i,t}^{\varepsilon }{(x},\cdot ,\cdot )$ is continuous
for a.e. $x\in \Omega ,$ for all $\varepsilon \in (0,1),$ $i=1,2$.

We recall from \cite[Proposition 9.64]{MMP2} that, for $t=0$ in $(\mathrm{P}%
_{t}^{\varepsilon })$, the decoupled system%
\begin{equation*}
(\mathrm{P}_{0}^{\varepsilon })\qquad \left\{ 
\begin{array}{l}
-\Delta _{p_{i}}u_{i}=\mathrm{F}_{i,0}^{\varepsilon }({x,}%
u_{1},u_{2})=\lambda _{1,p_{i}}(\mathcal{T}_{i,0}(u_{i}^{+}))^{p_{i}-1}+1%
\text{ in }\Omega \\ 
u_{i}=0\text{ \ on }\partial \Omega ,\text{ }i=1,2%
\end{array}%
\right.
\end{equation*}%
does not admit solutions $(u_{1},u_{2})$ in $W_{0}^{1,p_{1}}(\Omega )\times
W_{0}^{1,p_{2}}(\Omega )$. On account of $(\mathrm{H}_{4})$, (\ref{19}), (%
\ref{20}), we derive the estimate%
\begin{equation*}
\begin{array}{l}
|\mathrm{F}_{i,t}^{\varepsilon }({x,}u_{1},u_{2})|\leq |f_{i}(x,\mathcal{T}%
_{1,\varepsilon }(u_{1}),\mathcal{T}_{2,\varepsilon }(u_{2}))|+\lambda
_{1,p_{i}}(\mathcal{T}_{i,0}(u_{i}^{+}))^{p_{i}-1}+1 \\ 
\leq M_{i}(1+|\mathcal{T}_{1,\varepsilon }(u_{1})|^{\alpha _{i}}+|\mathcal{T}%
_{2,\varepsilon }(u_{2})|^{\beta _{i}})+\lambda
_{1,p_{i}}(u_{i}^{+})^{p_{i}-1}+1 \\ 
\leq C_{\varepsilon }(1+(u_{i}^{+})^{p_{i}-1}),\text{ for a.e. }x\in \Omega ,%
\text{ }i=1,2,%
\end{array}%
\end{equation*}%
for certain constant $C_{\varepsilon }>0$. Then, according to \cite[%
Corollary 8.13]{MMP2}, we conclude that each solution $(u_{1,\varepsilon
},u_{2,\varepsilon })$ of $(\mathrm{P}_{t}^{\varepsilon })$ belongs to $%
\mathcal{C}^{1}(\overline{\Omega })\times \mathcal{C}^{1}(\overline{\Omega }%
) $ and there exists a constant $R_{\varepsilon }>0$ such that 
\begin{equation}
\left\Vert u_{\varepsilon }\right\Vert _{\mathcal{C}^{1}(\overline{\Omega }%
)}<R_{\varepsilon },  \label{21}
\end{equation}%
for all $t\in (0,1]$ and $\varepsilon \in (0,1)$.

For every $\varepsilon \in (0,1)$, let us define the homotopy $\mathcal{H}%
_{\varepsilon }$ on $[0,1]\times \mathcal{B}_{R_{\varepsilon
}}(0)\rightarrow L^{p_{1}}(\Omega )\times L^{p_{2}}(\Omega )$ by%
\begin{equation*}
\mathcal{H}_{\varepsilon }(t,u_{1},u_{2})=I(u_{1},u_{2})-\left( 
\begin{array}{cc}
(-\Delta _{p_{1}})^{-1} & 0 \\ 
0 & (-\Delta _{p_{2}})^{-1}%
\end{array}%
\right) \left( 
\begin{array}{l}
\mathrm{F}_{1,t}^{\varepsilon }({x,}u_{1},u_{2}) \\ 
\multicolumn{1}{c}{\mathrm{F}_{2,t}^{\varepsilon }({x,}u_{1},u_{2})}%
\end{array}%
\right) .
\end{equation*}%
that is admissible for the Leray-Schauder topological degree by (\ref{21}),
the continuity of $\mathrm{F}{_{i,t}^{\varepsilon }(x},\cdot ,\cdot )$ for
a.e. $x\in \Omega $ and because the operator $(-\Delta _{p_{i}})^{-1}$ with
values in $L^{p_{i}}(\Omega )$ is compact.

Note that $(u_{1,\varepsilon },u_{2,\varepsilon })\in \mathcal{B}%
_{R_{\varepsilon }}(0)$ is a solution for $(\mathrm{P}^{\varepsilon })$ if,
and only if, 
\begin{equation*}
\begin{array}{c}
(u_{1,\varepsilon },u_{2,\varepsilon })\in \mathcal{B}_{R_{\varepsilon
}}(0)\,\,\,\mbox{and}\,\,\,\mathcal{H}_{\varepsilon }(1,u_{1,\varepsilon
},u_{2,\varepsilon })=0.%
\end{array}%
\end{equation*}%
The a priori estimate (\ref{21}) establishes expressly that solutions of $(%
\mathrm{P}_{t}^{\varepsilon })$ must lie in $\mathcal{B}_{R_{\varepsilon
}}(0)$, while the nonexistence of solutions to problem $(\mathrm{P}%
_{0}^{\varepsilon })$ yields 
\begin{equation*}
\deg \left( \mathcal{H}_{\varepsilon }(0,\cdot ,\cdot ),\mathcal{B}
_{R_{\varepsilon }}(0),0\right) =0,\text{ for all }\varepsilon \in (0,1).
\end{equation*}%
Consequently, the homotopy invariance property of the degree implies that 
\begin{equation}
\begin{array}{c}
\deg \left( \mathcal{H}_{\varepsilon }(1,\cdot ,\cdot ),\mathcal{B}
_{R_{\varepsilon }}(0),0\right) =0,\text{ for all }\varepsilon \in (0,1).%
\end{array}
\label{22}
\end{equation}

\subsubsection{\textbf{The degree on }$\mathcal{B}_{R_{\protect\varepsilon %
}}(0)\backslash \overline{\mathcal{B}}_{\protect\underline{u}}(0)$\textbf{.}}

We show that the degree of an operator corresponding to problem $(\mathrm{P}%
^{\varepsilon })$ is $1$ outside the ball $\mathcal{B}_{\underline{u}}(0)$.
To this end, let us define the problem%
\begin{equation*}
(\widehat{\mathrm{P}}_{t}^{\varepsilon })\qquad \left\{ 
\begin{array}{l}
-\Delta _{p_{i}}u_{i}=\widehat{\mathrm{F}}{_{i,t}^{\varepsilon }}({x,}%
u_{1},u_{2})\text{ in }\Omega \\ 
u_{i}=0\text{ \ on }\partial \Omega ,\text{ }i=1,2,%
\end{array}%
\right.
\end{equation*}%
for $t\in \lbrack 0,1]$ and $\varepsilon \in (0,1)$, where%
\begin{equation*}
\widehat{\mathrm{F}}{_{i,t}^{\varepsilon }}({x,}u_{1},u_{2}):=tf_{i}(x,%
\mathcal{T}_{1,\varepsilon }(u_{1}),\mathcal{T}_{2,\varepsilon
}(u_{2}))+(1-t)(\frac{2}{3})^{p_{i}-1}\lambda _{1,p_{i}}|\hat{\chi}_{\phi
_{1,p_{i}}}(u_{i})|^{p_{i}-2}\hat{\chi}_{\phi _{1,p_{i}}}(u_{i}),
\end{equation*}%
where $\hat{\chi}_{\phi _{1,p}}$ is a truncation defined by%
\begin{equation}
\hat{\chi}_{\phi _{1,p_{i}}}(s)=\left\{ 
\begin{array}{ll}
\frac{3}{2}s & \text{if }s\geq \phi _{1,p_{i}} \\ 
(\frac{1}{2}+sgn(s))\text{ }\phi _{1,p_{i}} & \text{if }-\phi _{1,p_{i}}\leq
s\leq \phi _{1,p_{i}} \\ 
\frac{1}{2}s & \text{if }s\leq -\phi _{1,p_{i}}.%
\end{array}%
\right. ,\text{ for }i=1,2.  \label{23}
\end{equation}%
Note from $(\mathrm{H}_{1})$ that every solution $(\hat{u}_{1,\varepsilon },%
\hat{u}_{2,\varepsilon })\in W_{0}^{1,p_{1}}(\Omega )\times
W_{0}^{1,p_{2}}(\Omega )$ of $(\widehat{\mathrm{P}}_{t}^{\varepsilon })$
satisfies $\hat{u}_{1,\varepsilon }(x),\hat{u}_{2,\varepsilon }(x)\neq 0$
for a.e. $x\in \Omega $. Therefore, this leads to conclude that $\mathrm{\ 
\tilde{F}}_{\varepsilon ,t}{(x},\cdot )$ is continuous for a.e. $x\in \Omega
,$ for all $\varepsilon \in (0,1)$.

We show that solutions of problem $(\widehat{\mathrm{P}}_{t}^{\varepsilon })$
cannot occur outside the ball $\mathcal{B}_{R_{\varepsilon }}(0)$.

\begin{proposition}
\label{P3} Assume that $(\mathrm{H}_{1})-(\mathrm{H}_{4})$ and (\ref{17})
are fulfilled.\ Then, any solution $(\hat{u}_{1},\hat{u}_{2})$ of $(\widehat{%
\mathrm{P}}_{t}^{\varepsilon })$ belongs to $\mathcal{C}^{1}(\overline{%
\Omega })\times \mathcal{C}^{1}(\overline{\Omega })$ and satisfy 
\begin{equation}
\left\Vert \hat{u}_{1}\right\Vert _{\mathcal{C}^{1}(\overline{\Omega }%
)},\left\Vert \hat{u}_{2}\right\Vert _{\mathcal{C}^{1}(\overline{\Omega }%
)}<R_{\varepsilon },  \label{24}
\end{equation}%
for $t\in \lbrack 0,1]$ and $\varepsilon \in (0,1)$. In addition, all
positive (resp. negative) solutions $(\hat{u}_{1,+},\hat{u}_{2,+})$ (resp. $(%
\hat{u}_{1,-},\hat{u}_{2,-})$) of $(\widehat{\mathrm{P}}_{t}^{\varepsilon })$
satisfy 
\begin{equation}
\begin{array}{c}
\underline{u}_{i}(x)\ll \hat{u}_{i,+}(x)\text{ \ (resp. }-\underline{u}%
_{i}(x)\gg \hat{u}_{i,-}(x)\text{)}\quad \forall x\in \Omega .%
\end{array}
\label{25}
\end{equation}
\end{proposition}

\begin{proof}
Let $(\hat{u}_{1},\hat{u}_{2})\in W_{0}^{1,p_{1}}(\Omega )\times
W_{0}^{1,p_{2}}(\Omega )$ be a solution of $(\widehat{\mathrm{P}}
_{t}^{\varepsilon })$. Using (\ref{20}) and for any $\theta \in 
\mathbb{R}
$, observe that 
\begin{equation*}
|\mathcal{T}_{i,\varepsilon }(\hat{u}_{i})|^{\theta }\leq \left\{ 
\begin{array}{ll}
(\frac{\varepsilon }{2})^{\theta } & \text{if }\theta <0 \\ 
(\frac{3\varepsilon }{2}+C||y_{i}||_{\infty })^{\theta } & \text{if }\theta
>0%
\end{array}
\right. ,\text{ for }i=1,2.
\end{equation*}
Moreover, on account of (\ref{23}), one has 
\begin{equation*}
(\frac{2}{3})^{p_{i}-1}|\hat{\chi}_{\phi _{1,p_{i}}}(\hat{u}_{i})|^{p_{i}-2} 
\hat{\chi}_{\phi _{1,p_{i}}}(\hat{u}_{i})\leq (\max \{\hat{u}_{i},\phi
_{1,p_{i}}\})^{p_{i}-1}.
\end{equation*}
Hence, replacing $\theta $ with either $\alpha _{i}$ or $\beta _{i}$ ($i=1,2$
) depending on the given situation, we obtain 
\begin{equation*}
\begin{array}{l}
|\widehat{\mathrm{F}}{_{i,t}^{\varepsilon }}({x,}\hat{u}_{1},\hat{u}
_{2})|\leq C_{\varepsilon }+\lambda _{1,p_{i}}(\max \{\hat{u}_{i},\phi
_{1,p_{i}}\})^{p_{i}-1},\text{ for all }\varepsilon \in (0,1),%
\end{array}%
\end{equation*}
where $C_{\varepsilon }$ is certain positive constant. Then,\ the regularity
theory up to the boundary (see \cite[Corollary 8.13]{MMP2}) together with
the compact embedding $\mathcal{C}^{1,\tau }(\overline{\Omega })\subset 
\mathcal{C}^{1}(\overline{\Omega })$ entails the bound in (\ref{24}), for
all $\varepsilon \in (0,1)$.

Next, we proceed to show the inequalities (\ref{25}). Let $(\hat{u}_{1,+}, 
\hat{u}_{2,+})$ be a positive solution of $(\widehat{\mathrm{P}}
_{t}^{\varepsilon }).$ By $(\mathrm{H}_{3})$ and (\ref{2}), we get 
\begin{equation}
\begin{array}{l}
\widehat{\mathrm{F}}{_{i,t}^{\varepsilon }}({x,}\hat{u}_{1,+},\hat{u}
_{2,+})\geq tm_{i}(\mathcal{T}_{1,\varepsilon }(\hat{u}_{1,+})^{\hat{\alpha}
_{i}}+\mathcal{T}_{2,\varepsilon }(\hat{u}_{2,+})^{\hat{\beta}
_{i}})+(1-t)\lambda _{1,p_{i}}\phi _{1,p_{i}}^{p_{i}-1} \\ 
\geq tm_{i}(\mathcal{T}_{1,\varepsilon }(\hat{u}_{1,+})^{\hat{\alpha}_{i}}+ 
\mathcal{T}_{2,\varepsilon }(\hat{u}_{2,+})^{\hat{\beta}_{i}})+(1-t)(\frac{2 
}{3})^{p_{i}-1}\lambda _{1,p_{i}}(c_{0}^{-1}d(x))^{p_{i}-1} \\ 
\geq tm_{i}(\mathcal{T}_{1,\varepsilon }(\hat{u}_{1,+})^{\hat{\alpha}_{i}}+ 
\mathcal{T}_{2,\varepsilon }(\hat{u}_{2,+})^{\hat{\beta}_{i}})+(1-t)\lambda
_{1,p_{i}}(c_{0}^{-1}\delta )^{p_{i}-1}\text{\ in\ }\Omega \backslash 
\overline{\Omega }_{\delta }.%
\end{array}
\label{26}
\end{equation}
From (\ref{20}) and for any $\theta \in 
\mathbb{R}
$, we have 
\begin{equation*}
|\mathcal{T}_{i,\varepsilon }(\hat{u}_{i,+})|^{\theta }\geq \left\{ 
\begin{array}{ll}
(\frac{3\varepsilon }{2}+C||y_{i}||_{\infty })^{\theta } & \text{if }\theta
<0 \\ 
(\frac{\varepsilon }{2})^{\theta } & \text{if }\theta >0%
\end{array}
\right. ,\text{ for }i=1,2.
\end{equation*}

Therefore, 
\begin{eqnarray*}
&&\mathcal{T}_{1,\varepsilon }(\hat{u}_{1,+})^{\hat{\alpha}_{i}}+\mathcal{T}
_{2,\varepsilon }(\hat{u}_{2,+})^{\hat{\beta}_{i}} \\
&\geq &\left\{ 
\begin{array}{ll}
(\frac{3\varepsilon }{2}+C||y_{i}||_{\infty })^{\hat{\alpha}_{i}} & \text{if 
}\hat{\alpha}_{i}<0 \\ 
(\frac{\varepsilon }{2})^{\hat{\alpha}_{i}} & \text{if }\hat{\alpha}_{i}>0%
\end{array}
\right. +\left\{ 
\begin{array}{ll}
(\frac{3\varepsilon }{2}+C||y_{i}||_{\infty })^{\hat{\beta}_{i}} & \text{if }
\hat{\beta}_{i}<0 \\ 
(\frac{\varepsilon }{2})^{\hat{\beta}_{i}} & \text{if }\hat{\beta}_{i}>0.%
\end{array}
\right.
\end{eqnarray*}
For $0>\hat{\alpha}_{1}>-(p_{1}-1)$ and $0<\hat{\beta}_{1}<p_{1}-1$ in $( 
\mathrm{H}_{3}),$ it readily seen that 
\begin{equation}
\begin{array}{c}
tm_{1}((\frac{3\varepsilon }{2}+C||y_{1}||_{\infty })^{\hat{\alpha}_{1}}+( 
\frac{\varepsilon }{2})^{\hat{\beta}_{1}})+(1-t)\lambda
_{1,p_{1}}(c_{0}^{-1}\delta )^{p_{1}-1} \\ 
>C^{-(p_{1}-1)}(\delta ^{\hat{\alpha}_{1}}+d(x)^{\hat{\beta}_{1}})\text{ in }
\Omega \backslash \overline{\Omega }_{\delta },%
\end{array}
\label{27}
\end{equation}
provided $C>1$ is large. The same inequality occurs in the corresponding
cases $0>\hat{\beta}_{2}>-(p_{2}-1)$ and $0<\hat{\alpha}_{2}<p_{2}-1.$ That
is, 
\begin{equation}
\begin{array}{c}
tm_{2}((\frac{\varepsilon }{2})^{\hat{\alpha}_{2}}+(\frac{3\varepsilon }{2}
+C||y_{2}||_{\infty })^{\hat{\beta}_{2}})+(1-t)\lambda
_{1,p_{2}}(c_{0}^{-1}\delta )^{p_{1}-1} \\ 
>C^{-(p_{2}-1)}(d(x)^{\hat{\alpha}_{2}}+\delta ^{\hat{\beta}_{2}})\text{ in }
\Omega \backslash \overline{\Omega }_{\delta }.%
\end{array}
\label{28}
\end{equation}
Then, by noting that 
\begin{equation}
d(x)^{\hat{\alpha}_{1}}+d(x)^{\hat{\beta}_{1}}\leq \delta ^{\hat{\alpha}
_{1}}+d(x)^{\hat{\beta}_{1}}  \label{29}
\end{equation}
and 
\begin{equation}
d(x)^{\hat{\alpha}_{2}}+d(x)^{\hat{\beta}_{2}}\leq d(x)^{\hat{\alpha}
_{2}}+\delta ^{\hat{\beta}_{2}},  \label{30}
\end{equation}
gathering (\ref{26})-(\ref{30}) together leads to 
\begin{equation}
\widehat{\mathrm{F}}{_{i,t}^{\varepsilon }}({x,}\hat{u}_{1,+},\hat{u}
_{2,+})>C^{-(p_{i}-1)}(d(x)^{\hat{\alpha}_{i}}+d(x)^{\hat{\beta}_{i}})\text{
\ in\ }\Omega \backslash \overline{\Omega }_{\delta }.  \label{31}
\end{equation}
In $\Omega _{\delta }$, it is readily apparent that 
\begin{equation}
\widehat{\mathrm{F}}{_{i,t}^{\varepsilon }}({x,}\hat{u}_{1,+},\hat{u}
_{2,+})>0>-C^{-(p_{i}-1)}\text{ in }\Omega _{\delta },  \label{32}
\end{equation}
for all $t\in \lbrack 0,1]$ and all $\varepsilon \in (0,1),i=1,2$.

Set the function $\underline{\mathrm{X}}_{i}:\Omega \rightarrow \mathbb{R}$
given by 
\begin{equation*}
\underline{\mathrm{X}}_{i}(x):=C^{-(p_{i}-1)}\left\{ 
\begin{array}{ll}
d(x)^{\hat{\alpha}_{i}}+d(x)^{\hat{\beta}_{i}} & \text{in \ }\Omega
\backslash \overline{\Omega }_{\delta }, \\ 
-1 & \text{in \ }\Omega _{\delta },%
\end{array}
\right.
\end{equation*}
and note that $\underline{\mathrm{X}}_{i}\in L^{\infty }(\Omega )$ ($i=1,2$
). In view of (\ref{31}) and (\ref{32}), for each compact set $\mathrm{K}
\subset \subset \Omega $, there is a constant $\tau =\tau (\mathrm{K})>0$
such that 
\begin{equation*}
\underline{\mathrm{X}}_{i}(x)+\tau <\widehat{\mathrm{F}}{_{i,t}^{\varepsilon
}}({x,}\hat{u}_{1,+},\hat{u}_{2,+})\text{ a.e. in }\Omega \cap \mathrm{K}.
\end{equation*}
Then, by (\ref{5}), (\ref{7}) and the strong comparison principle \cite[
Proposition $2.6$ and Remark $2.8$]{AR}, we infer that $\underline{u}
_{i}(x)\ll \hat{u}_{i,+}(x)$ in $\overline{\Omega }$. A quite similar
argument shows that $-\underline{u}_{i}(x)\gg \hat{u}_{i,-}(x)$ in $%
\overline{\Omega }$. This ends the proof.
\end{proof}

Let us define the homotopy $\mathcal{N}_{\varepsilon }$ on $[0,1]\times 
\mathcal{B}_{R_{\varepsilon }}(0)\backslash \overline{\mathcal{B}}_{%
\underline{u}}(0)\rightarrow L^{p_{1}}(\Omega )\times L^{p_{2}}(\Omega )$ by%
\begin{equation}
\mathcal{N}_{\varepsilon }(t,u_{1},u_{2})=I(u_{1},u_{2})-\left( 
\begin{array}{cc}
(-\Delta _{p_{1}})^{-1} & 0 \\ 
0 & (-\Delta _{p_{2}})^{-1}%
\end{array}%
\right) \left( 
\begin{array}{l}
\widehat{\mathrm{F}}{_{1,t}^{\varepsilon }}({x,}u_{1},u_{2}) \\ 
\multicolumn{1}{c}{\widehat{\mathrm{F}}{_{2,t}^{\varepsilon }}({x,}%
u_{1},u_{2})}%
\end{array}%
\right) .  \label{33}
\end{equation}%
for $t\in \lbrack 0,1]$ and $\varepsilon \in (0,1)$. Clearly, $\mathcal{N}%
_{\varepsilon }$ is well defined, compact and continuous a.e. in $\Omega $.
Moreover, $(u_{1},u_{2})\in \mathcal{B}_{R_{\varepsilon }}(0)\backslash 
\overline{\mathcal{B}}_{\underline{u}}(0)$ is a solution of system $(\mathrm{%
\ P}^{\varepsilon })$ if, and only if, 
\begin{equation*}
\begin{array}{c}
(u_{1},u_{2})\in \mathcal{B}_{R_{\varepsilon }}(0)\backslash \overline{%
\mathcal{B}}_{\underline{u}}(0)\,\,\,\mbox{and}\,\,\,\mathcal{N}%
_{\varepsilon }(1,u_{1},u_{2})=0.%
\end{array}%
\end{equation*}%
In view of (\ref{8}), $\phi _{1,p_{i}}\in \mathcal{B}_{R_{\varepsilon
}}(0)\backslash \overline{\mathcal{B}}_{\underline{u}}(0)$ which, by (\ref%
{23}) and (\ref{1}), is actually the unique solution of the problem 
\begin{equation*}
-\Delta _{p_{i}}u_{i}=(\frac{2}{3})^{p_{i}-1}\lambda _{1,p_{i}}|\hat{\chi}%
_{\phi _{1,p_{i}}}(u_{i})|^{p_{i}-2}\hat{\chi}_{\phi _{1,p_{i}}}(u_{i})\text{
in }\Omega ,\text{ }u_{i}=0\text{ on }\partial \Omega .
\end{equation*}%
Then, the homotopy invariance property of the degree gives 
\begin{equation}
\begin{array}{l}
\deg (\mathcal{N}_{\varepsilon }(1,\cdot ,\cdot ),\mathcal{B}%
_{R_{\varepsilon }}(0)\backslash \overline{\mathcal{B}}_{\underline{u}}(0),0)
\\ 
=\deg (\mathcal{N}_{\varepsilon }(0,\cdot ,\cdot ),\mathcal{B}%
_{R_{\varepsilon }}(0)\backslash \overline{\mathcal{B}}_{\underline{u}%
}(0),0)=1.%
\end{array}
\label{34}
\end{equation}%
Since 
\begin{equation*}
\mathcal{H}_{\varepsilon }(1,\cdot ,\cdot )=\mathcal{N}_{\varepsilon
}(1,\cdot ,\cdot )\,\,\,\text{in}\,\,\,\mathcal{B}_{R_{\varepsilon
}}(0)\backslash \overline{\mathcal{B}}_{\underline{u}}(0),\text{ for all }%
\varepsilon \in (0,1),
\end{equation*}%
we deduce that 
\begin{equation}
\begin{array}{c}
\deg (\mathcal{H}_{\varepsilon }(1,\cdot ,\cdot ),\mathcal{B}%
_{R_{\varepsilon }}(0)\backslash \overline{\mathcal{B}}_{\underline{u}%
}(0),0)=1.%
\end{array}
\label{35}
\end{equation}

\subsubsection{\textbf{Proof of Theorem \protect\ref{T3}.}}

Hereafter, we will assume that $\mathcal{H}_{\varepsilon
}(1,u_{1},u_{2})\not=0,$ for all $(u_{1},u_{2})\in \partial \mathcal{B}_{%
\underline{u}}(0),$ for all $\varepsilon \in (0,1).$ Otherwise, $%
(u_{1},u_{2})\in \partial \mathcal{B}_{\underline{u}}(0)$ would be a
solution of $(\mathrm{P}^{\varepsilon })$ within $\left[ -\underline{u}_{1},%
\underline{u}_{1}\right] \times \left[ -\underline{u}_{2},\underline{u}_{2}%
\right] $ and thus, Theorem \ref{T3} is proved.

By virtue of the domain additivity property of Leray-Schauder degree it
follows that 
\begin{equation*}
\begin{array}{c}
\deg (\mathcal{H}_{\varepsilon }(1,\cdot ,\cdot ),\mathcal{B}%
_{R_{\varepsilon }}(0)\backslash \overline{\mathcal{B}}_{\underline{u}%
}(0),0)+\deg (\mathcal{H}_{\varepsilon }(1,\cdot ,\cdot ),\mathcal{B}_{%
\underline{u}}(0),0) \\ 
=\deg (\mathcal{H}_{\varepsilon }(1,\cdot ,\cdot ),\mathcal{B}%
_{R_{\varepsilon }}(0),0).%
\end{array}%
\end{equation*}%
Hence, by (\ref{22}) and (\ref{35}), we deduce that $\deg (\mathcal{H}%
_{\varepsilon }(1,\cdot ,\cdot ),\mathcal{B}_{\underline{u}}(0),0)=-1,$
showing that problem $(\mathrm{P}^{\varepsilon })$ has a solution $%
(u_{1,\varepsilon },u_{2,\varepsilon })\in \mathcal{B}_{\underline{u}}(0),$
for all $\varepsilon \in (0,1)$. The nonlinear regularity theory \cite{L}
guarantees that $(u_{1,\varepsilon },u_{2,\varepsilon })\in \mathcal{C}%
^{1,\tau }(\overline{\Omega })\times \mathcal{C}^{1,\tau }(\overline{\Omega }%
)$ for certain $\tau \in (0,1)$.

\subsection{\textbf{Proof of Theorem \protect\ref{T2}.}}

Set $\varepsilon =\frac{1}{n}$ in $(\mathrm{P}^{\varepsilon })$ with any
positive integer $n\geq 1.$ According to Theorem \ref{T3}, there exists $%
(u_{1,n},u_{2,n}):=(u_{1,\frac{1}{n}},u_{2,\frac{1}{n}})\in \mathcal{C}%
^{1,\tau }(\overline{\Omega })\times \mathcal{C}^{1,\tau }(\overline{\Omega }%
)$ solution of $(\mathrm{P}^{n})$ ($(\mathrm{P}^{\varepsilon })$ with $%
\varepsilon =\frac{1}{n}$) such that 
\begin{equation}
u_{i,n}\in \lbrack -\underline{u}_{i},\underline{u}_{i}]  \label{36}
\end{equation}%
and 
\begin{equation}
\begin{array}{l}
\int_{\Omega }|\nabla u_{i,n}|^{p_{i}-2}\nabla u_{i,n}\text{\thinspace }%
\nabla \varphi _{i}\text{ }\mathrm{d}x \\ 
=\int_{\Omega }f_{i}(x,u_{1,n}+\gamma _{n}(u_{1,n}),u_{2,n}+\gamma
_{n}(u_{2,n}))\varphi _{i}\text{ }\mathrm{d}x%
\end{array}
\label{37}
\end{equation}%
for all $\varphi _{i}\in W_{0}^{1,p_{i}}(\Omega ),$ $i=1,2$, where $\gamma
_{n}(\cdot ):=\gamma _{\frac{1}{n}}(\cdot )$. Passing to relabeled
subsequences, the compact embedding $\mathcal{C}^{1,\tau }(\overline{\Omega }%
)\hookrightarrow \mathcal{C}^{1}(\overline{\Omega })$ entails the strong
convergence%
\begin{equation}
u_{i,n}\rightarrow u_{i}^{\ast }\;\text{\ in }\;\mathcal{C}^{1}(\overline{%
\Omega })  \label{38}
\end{equation}%
and therefore%
\begin{equation}
u_{i,n}\rightarrow u_{i}^{\ast }\;\text{\ in }\;W_{0}^{1,p_{i}}(\Omega ),
\label{39}
\end{equation}%
which, by Lebesgue's dominated convergence theorem, leads to%
\begin{equation}
\lim_{n\rightarrow +\infty }\int_{\Omega }|\nabla u_{i,n}|^{p_{i}-2}\nabla
u_{i,n}\text{\thinspace }\nabla \varphi _{i}\text{ }\mathrm{d}x=\int_{\Omega
}|\nabla u_{i}^{\ast }|^{p_{i}-2}\nabla u_{i}^{\ast }\text{\thinspace }%
\nabla \varphi _{i}\text{ }\mathrm{d}x,  \label{40}
\end{equation}%
for all $\varphi _{i}\in W_{0}^{1,p_{i}}(\Omega )$, $i=1,2.$ Moreover, Young
inequality implies that%
\begin{equation}
\int_{\Omega }|\nabla u_{i}^{\ast }|^{p_{i}-2}\nabla u_{i}^{\ast }\text{
\thinspace }\nabla \varphi _{i}\text{ }\mathrm{d}x\leq \frac{p_{i}-1}{p_{i}}%
\left\Vert \nabla u_{i}^{\ast }\right\Vert _{p_{i}}^{p_{i}}+\frac{1}{p_{i}}%
\left\Vert \varphi _{i}\right\Vert _{p_{i}}^{p_{i}}.  \label{41}
\end{equation}%
Next, we show that 
\begin{equation}
\begin{array}{l}
\lim_{n\rightarrow +\infty }\int_{\Omega }f_{i}(x,u_{1,n}+\gamma
_{n}(u_{1,n}),u_{2,n}+\gamma _{n}(u_{2,n}))\varphi _{i}\text{ }\mathrm{d}x
\\ 
=\int_{\Omega }f_{i}(x,u_{1}^{\ast },u_{2}^{\ast })\varphi _{i}\text{ }%
\mathrm{d}x,\text{ for }i=1,2,%
\end{array}
\label{42}
\end{equation}%
for all $\varphi _{i}\in W_{0}^{1,p_{i}}(\Omega ),$ for $i=1,2.$ Assume $%
\varphi _{i}\geq 0$ in $\Omega $, for $i=1,2$ and write%
\begin{equation*}
\int_{\Omega }f_{i}(x,u_{1}^{\ast },u_{2}^{\ast })\varphi _{i}\text{ }%
\mathrm{d}x=\int_{\Omega }f_{i}^{+}(x,u_{1}^{\ast },u_{2}^{\ast })\varphi
_{i}\text{ }\mathrm{d}x-\int_{\Omega }f_{i}^{-}(x,u_{1}^{\ast },u_{2}^{\ast
})\varphi _{i}\text{ }\mathrm{d}x.
\end{equation*}%
By Fatou's Lemma, along with (\ref{38}), given that $f_{i}(x,s,t)$ is a
Carath\'{e}odory function for $(x,s,t)\in \Omega \times (%
\mathbb{R}
\backslash \{0\})^{2}$, one has%
\begin{equation*}
\begin{array}{l}
\int_{\Omega }f_{i}^{+}(x,u_{1}^{\ast },u_{2}^{\ast })\varphi _{i}\text{ }%
\mathrm{d}x \\ 
\leq \int_{\Omega }\lim_{n\rightarrow +\infty }\inf
(f_{i}^{+}(x,u_{1,n}+\gamma _{n}(u_{1,n}),u_{2,n}+\gamma
_{n}(u_{2,n}))\varphi _{i})\text{ }\mathrm{d}x \\ 
\leq \lim_{n\rightarrow +\infty }\inf \int_{\Omega
}f_{i}^{+}(x,u_{1,n}+\gamma _{n}(u_{1,n}),u_{2,n}+\gamma
_{n}(u_{2,n}))\varphi _{i}\text{ }\mathrm{d}x,%
\end{array}%
\end{equation*}%
as well as%
\begin{equation*}
\begin{array}{l}
\int_{\Omega }-f_{i}^{-}(x,u_{1}^{\ast },u_{2}^{\ast })\varphi _{i}\text{ }%
\mathrm{d}x \\ 
\leq \int_{\Omega }\lim_{n\rightarrow +\infty }\sup
(-f_{i}^{-}(x,u_{1,n}+\gamma _{n}(u_{1,n}),u_{2,n}+\gamma
_{n}(u_{2,n}))\varphi _{i})\text{ }\mathrm{d}x \\ 
\leq \lim_{n\rightarrow +\infty }\sup \int_{\Omega
}-f_{i}^{-}(x,u_{1,n}+\gamma _{n}(u_{1,n}),u_{2,n}+\gamma
_{n}(u_{2,n}))\varphi _{i}\text{ }\mathrm{d}x.%
\end{array}%
\end{equation*}%
Then, using (\ref{37}), (\ref{41}) and (\ref{40}), it follows that%
\begin{equation*}
\begin{array}{l}
\int_{\Omega }f_{i}(x,u_{1}^{\ast },u_{2}^{\ast })\varphi _{i}\text{ }%
\mathrm{d}x \\ 
\leq \lim_{n\rightarrow +\infty }\inf \int_{\Omega
}f_{i}^{+}(x,u_{1,n}+\gamma _{n}(u_{1,n}),u_{2,n}+\gamma
_{n}(u_{2,n}))\varphi _{i}\text{ }\mathrm{d}x \\ 
\text{ \ \ }+\lim_{n\rightarrow +\infty }\sup \int_{\Omega
}-f_{i}^{-}(x,u_{1,n}+\gamma _{n}(u_{1,n}),u_{2,n}+\gamma
_{n}(u_{2,n}))\varphi _{i}\text{ }\mathrm{d}x \\ 
\leq \lim_{n\rightarrow +\infty }\int_{\Omega }f_{i}^{+}(x,u_{1,n}+\gamma
_{n}(u_{1,n}),u_{2,n}+\gamma _{n}(u_{2,n}))\varphi _{i}\text{ }\mathrm{d}x
\\ 
\text{ \ \ }+\lim_{n\rightarrow +\infty }\int_{\Omega
}-f_{i}^{-}(x,u_{1,n}+\gamma _{n}(u_{1,n}),u_{2,n}+\gamma
_{n}(u_{2,n}))\varphi _{i}\text{ }\mathrm{d}x \\ 
=\lim_{n\rightarrow +\infty }\int_{\Omega }f_{i}(x,u_{1,n}+\gamma
_{n}(u_{1,n}),u_{2,n}+\gamma _{n}(u_{2,n}))\varphi _{i}\text{ }\mathrm{d}x
\\ 
\leq \frac{p_{i}-1}{p_{i}}\left\Vert \nabla u_{i}^{\ast }\right\Vert
_{p_{i}}^{p_{i}}+\frac{1}{p_{i}}\left\Vert \varphi _{i}\right\Vert
_{p_{i}}^{p_{i}},%
\end{array}%
\end{equation*}%
showing that%
\begin{equation}
f_{i}(x,u_{1}^{\ast },u_{2}^{\ast })\varphi _{i}\in L^{1}(\Omega ),
\label{43}
\end{equation}%
for all $\varphi _{i}\in W_{0}^{1,p_{i}}(\Omega )$ with $\varphi _{i}\geq 0$
in $\Omega ,$ $i=1,2.$

For a fixed $\mu >0$, we write 
\begin{equation}
\begin{array}{l}
\int_{\Omega }f_{i}(x,u_{1,n}+\gamma _{n}(u_{1,n})),u_{2,n}+\gamma
_{n}(u_{2,n}))\varphi _{i}\text{ }\mathrm{d}x \\ 
=\int_{\Omega \cap \{|u_{i,n}|\leq \mu \}}f_{i}(x,u_{1,n}+\gamma
_{n}(u_{1,n})),u_{2,n}+\gamma _{n}(u_{2,n}))\varphi _{i}\text{ }\mathrm{d}x
\\ 
+\int_{\Omega \cap \{|u_{i,n}|>\mu \}}f_{i}(x,u_{1,n}+\gamma
_{n}(u_{1,n})),u_{2,n}+\gamma _{n}(u_{2,n}))\varphi _{i}\text{ }\mathrm{d}x.%
\end{array}
\label{44}
\end{equation}%
Define the truncation $\chi _{\mu }:%
\mathbb{R}
\rightarrow \lbrack 0,+\infty \lbrack $ by 
\begin{equation*}
\chi _{\mu }(s)=\left\{ 
\begin{array}{ll}
0 & \text{if }|s|\geq 2\mu , \\ 
2-sgn(s)\frac{s}{\mu } & \text{if }\mu \leq |s|\leq 2\mu , \\ 
1 & \text{if }|s|\leq \mu ,%
\end{array}
\right.
\end{equation*}%
Test in (\ref{37}) with $\chi _{\mu }(u_{i,n}^{+})\varphi _{i}\in
W_{0}^{1,p_{i}}(\Omega ),$ which is possible due to the continuity of
function $\chi _{\mu },$ reads as%
\begin{equation}
\begin{array}{l}
\int_{\Omega }|\nabla u_{i,n}|^{p_{i}-2}\nabla u_{i,n}\text{\thinspace }
\nabla (\chi _{\mu }(u_{i,n}^{+})\varphi _{i})\text{ }\mathrm{d}x \\ 
=\int_{\Omega }f_{i}(x,u_{1,n}+\gamma _{n}(u_{1,n})),u_{2,n}+\gamma
_{n}(u_{2,n}))\chi _{\mu }(u_{i,n}^{+})\varphi _{i}\text{ }\mathrm{d}x.%
\end{array}
\label{45}
\end{equation}%
By definition of $\chi _{\mu }$ we get 
\begin{equation}
\int_{\Omega }|\nabla u_{i,n}|^{p_{i}}\chi _{\mu }^{\prime
}(u_{i,n}^{+})\varphi _{i}\text{ }\mathrm{d}x=-\frac{1}{\mu }\int_{\Omega
}|\nabla u_{i,n}|^{p_{i}}\varphi _{i}\text{ }\mathrm{d}x.  \label{46}
\end{equation}%
Thence%
\begin{equation}
\int_{\Omega }|\nabla u_{i,n}|^{p_{i}-2}\nabla u_{i,n}\text{\thinspace }
\nabla (\chi _{\mu }(u_{i,n}^{+})\varphi _{i})\text{ }\mathrm{d}x\leq
\int_{\Omega }|\nabla u_{i,n}|^{p_{i}-2}\nabla u_{i,n}\text{\thinspace }
\nabla \varphi _{i}\text{ }\chi _{\mu }(u_{i,n}^{+})\text{ }\mathrm{d}x,
\label{47}
\end{equation}%
which, by (\ref{39}) together with Lebesgue's Theorem, gives%
\begin{equation}
\lim_{n\rightarrow +\infty }\int_{\Omega }|\nabla u_{i,n}|^{p_{i}-2}\nabla
u_{i,n}\text{\thinspace }\nabla \varphi _{i}\text{ }\chi _{\mu
}(u_{i,n}^{+}) \text{ }\mathrm{d}x\leq \int_{\Omega }|\nabla u_{i}^{\ast
}|^{p_{i}-2}\nabla u_{i}^{\ast }\text{\thinspace }\nabla \varphi _{i}\text{ }%
\chi _{\mu }(u_{i}^{\ast }{}^{+})\text{ }\mathrm{d}x.  \label{48}
\end{equation}%
Repeating the previous argument by testing in (\ref{37}) with $\chi _{\mu
}(-u_{i,n}^{-})\varphi _{i}\in W_{0}^{1,p_{i}}(\Omega )$, we get%
\begin{equation}
\lim_{n\rightarrow +\infty }\int_{\Omega }|\nabla u_{i,n}|^{p_{i}-2}\nabla
u_{i,n}\text{\thinspace }\nabla \varphi _{i}\text{ }\chi _{\mu
}(-u_{i,n}^{-})\text{ }\mathrm{d}x\leq \int_{\Omega }|\nabla u_{i}^{\ast
}|^{p_{i}-2}\nabla u_{i}^{\ast }\text{\thinspace }\nabla \varphi _{i}\text{ }
\chi _{\mu }(-u_{i}^{\ast }{}^{-})\text{ }\mathrm{d}x.  \label{49}
\end{equation}%
Note from the definition of $\chi _{\mu }(\cdot )$ that%
\begin{equation*}
\chi _{\mu }(-u_{i,n}^{-})+\chi _{\mu }(u_{i,n}^{+})=\chi _{\mu }(u_{i,n})
\end{equation*}%
and%
\begin{equation*}
\chi _{\mu }(-u_{i}^{\ast }{}^{-})+\chi _{\mu }(u_{i}^{\ast }{}^{+})=\chi
_{\mu }(u_{i}^{\ast }).
\end{equation*}%
Then, in view of (\ref{44}), (\ref{45})-(\ref{49}), and for $\varphi _{i}\in
W_{0}^{1,p_{i}}(\Omega ),$ $\varphi _{i}\geq 0$ in $\Omega $, we get%
\begin{eqnarray*}
&&\lim_{n\rightarrow +\infty }\int_{\Omega \cap \{|u_{i,n}|\leq \mu
\}}f_{i}(x,u_{1,n}+\gamma _{n}(u_{1,n})),u_{2,n}+\gamma
_{n}(u_{2,n}))\varphi _{i}\text{ }\mathrm{d}x \\
&=&\lim_{n\rightarrow +\infty }\int_{\Omega \cap \{|u_{i,n}|\leq \mu
\}}f_{i}(x,u_{1,n}+\gamma _{n}(u_{1,n})),u_{2,n}+\gamma
_{n}(u_{2,n}))\varphi _{i}\text{ }\chi _{\mu }(u_{i,n})\text{ }\mathrm{d}x \\
&\leq &\int_{\Omega }|\nabla u_{i}^{\ast }|^{p_{i}-2}\nabla u_{i}^{\ast } 
\text{\thinspace }\nabla \varphi _{i}\text{ }\chi _{\mu }(u_{i}^{\ast }) 
\text{ }\mathrm{d}x.
\end{eqnarray*}%
Since%
\begin{equation*}
|\nabla u_{i}^{\ast }|^{p_{i}-2}\nabla u_{i}^{\ast }\text{\thinspace }\nabla
u_{i}^{\ast }\text{ }\chi _{\mu }(u_{i}^{\ast })\rightarrow 0\text{ a.e. in }
\Omega ,\text{ as }\mu \rightarrow 0,
\end{equation*}%
Lebesgue's Theorem implies that 
\begin{equation}
\lim_{\mu \rightarrow 0}\lim_{n\rightarrow +\infty }\int_{\Omega \cap
\{|u_{i,n}|\leq \mu \}}f_{i}(x,u_{1,n}+\gamma _{n}(u_{1,n})),u_{2,n}+\gamma
_{n}(u_{2,n}))\varphi _{i}\text{ }\mathrm{d}x=0.  \label{50}
\end{equation}

On the other hand, noting that 
\begin{eqnarray*}
&&\int_{\Omega \cap \{|u_{i,n}|>\mu \}}f_{i}(x,u_{1,n}+\gamma
_{n}(u_{1,n})),u_{2,n}+\gamma _{n}(u_{2,n}))\varphi _{i}\text{ }\mathrm{d}x
\\
&=&\int_{\Omega }f_{i}(x,u_{1,n}+\gamma _{n}(u_{1,n})),u_{2,n}+\gamma
_{n}(u_{2,n}))\varphi _{i}\text{ }\mathbbm{1}_{\{|u_{i,n}|>\mu \}}\mathrm{d}x
\end{eqnarray*}%
and 
\begin{equation*}
\mathbbm{1}_{\{|u_{i,n}|>\mu \}}\rightarrow \mathbbm{1}_{\{|u_{i}^{\ast
}|>\mu \}}\text{ a.e. on }\{x\in \Omega _{i,+}:|u_{i,n}|\neq \mu \},\text{
for }i=1,2.
\end{equation*}%
By (\ref{38}), (\ref{43}), $(\mathrm{H}_{4})$ and (\ref{17}), together with
Lebesgue's Theorem, it follows that%
\begin{equation}
\begin{array}{l}
\lim_{n\rightarrow +\infty }\int_{\Omega \cap \{|u_{i,n}|>\mu
\}}f_{i}(x,u_{1,n}+\gamma _{n}(u_{1,n})),u_{2,n}+\gamma
_{n}(u_{2,n}))\varphi _{i}\text{ }\mathrm{d}x \\ 
=\int_{\Omega \cap \{|u_{i}^{\ast }|>\mu \}}f_{i}(x,u_{1}^{\ast
},u_{2}^{\ast })\varphi _{i}\text{ }\mathrm{d}x.%
\end{array}
\label{51}
\end{equation}%
From (\ref{43}) and the fact that%
\begin{equation*}
\mathbbm{1}_{\{|u_{i,n}|>\mu \}}\rightarrow \mathbbm{1}_{\{|u_{i}^{\ast
}|>0\}}\text{ a.e. in }\Omega \text{, as }\mu \rightarrow 0,
\end{equation*}%
because the set $\{x\in \Omega :|u_{i}^{\ast }(x)|=\mu \}$ is negligible, we
infer that%
\begin{equation}
\begin{array}{c}
\lim_{\mu \rightarrow 0}\int_{\Omega \cap \{|u_{i}^{\ast }|>\mu
\}}f_{i}(x,u_{1}^{\ast },u_{2}^{\ast })\varphi _{i}\text{ }\mathrm{d}%
x=\int_{\Omega \cap \{|u_{i}^{\ast }|>0\}}f_{i}(x,u_{1}^{\ast },u_{2}^{\ast
})\varphi _{i}\text{ }\mathrm{d}x \\ 
=\int_{\Omega }f_{i}(x,u_{1}^{\ast },u_{2}^{\ast })\varphi _{i}\text{ }%
\mathrm{d}x.%
\end{array}
\label{52}
\end{equation}%
Hence, gathering (\ref{44}), (\ref{50}) and (\ref{52}) together we deduce
that (\ref{42}) is fulfilled for all $\varphi _{i}\in W_{0}^{1,p_{i}}(\Omega
)$ with $\varphi _{i}\geq 0$ in $\Omega $.

Finally, writing $\varphi _{i}=\varphi _{i}^{+}-\varphi _{i}^{-}$ for $%
\varphi _{i}\in W_{0}^{1,p_{i}}(\Omega )$ and bearing in mind the linearity
property of (\ref{42}) in $\varphi _{i}$, we conclude that (\ref{42}) holds
for every $\varphi _{i}\in W_{0}^{1,p_{i}}(\Omega )$. Consequently, on
account of (\ref{40}) and (\ref{42}), we may pass to the limit in (\ref{37})
to conclude that $(u_{1}^{\ast },u_{2}^{\ast })\in W_{0}^{1,p_{1}}(\Omega
)\times W_{0}^{1,p_{2}}(\Omega )$ is a solution of problem $(\mathrm{P})$
within $[-\underline{u}_{1},\underline{u}_{1}]\times \lbrack -\underline{u}%
_{2},\underline{u}_{2}]$. Theorem \ref{T1} forces that $(u_{1}^{\ast
},u_{2}^{\ast })$ is nodal in the sens that the components $u_{1}^{\ast }$
and $u_{2}^{\ast }$ are nontrivial and at least are not of the same constant
sign. This completes the proof.

\subsection{Nodal solutions with precise sign information}

Nodal solutions for $(\mathrm{P})$ where both components are sign changing
are provided in the next theorem.

\begin{theorem}
\label{T4} Assume $(\mathrm{H}_{1})-(\mathrm{H}_{5})$\textrm{\ }and (\ref{17}
) are fulfilled. Then, problem $(\mathrm{P})$ admits a nodal solution $%
(u_{1}^{\ast },u_{2}^{\ast })$ in $W_{0}^{1,p_{1}}(\Omega )\times
W_{0}^{1,p_{2}}(\Omega ),$ where both components $u_{1}^{\ast }$ and $%
u_{2}^{\ast }$ are sign changing. Moreover, $u_{1}^{\ast }$ and $u_{2}^{\ast
}$ change sign simultaneously, that is, 
\begin{equation}
u_{1}^{\ast }u_{2}^{\ast }\geq 0.  \label{55}
\end{equation}
\end{theorem}

\begin{proof}
Recall from Theorem \ref{T2} that problem $(\mathrm{P})$ admits nodal
solutions $(u_{1}^{\ast },u_{2}^{\ast })$ in $W_{0}^{1,p_{1}}(\Omega )\times
W_{0}^{1,p_{2}}(\Omega )$ where the nontrivial components $u_{1}^{\ast }$
and $u_{2}^{\ast }$ are not of the same constant sign. Assume that $%
u_{1}^{\ast }<0<u_{2}^{\ast }$. Test the first equation in $(\mathrm{P})$ by 
$-u_{1}^{\ast }{}^{-}$, in view of $(\mathrm{H}_{5})$, we get 
\begin{equation*}
\int_{\Omega }|\nabla u_{1}^{\ast -}|^{p_{1}}\text{ }\mathrm{d}%
x=-\int_{\Omega }f_{1}(x,u_{1}^{\ast },u_{2}^{\ast })u_{1}^{\ast }{}^{-}%
\text{ }\mathrm{d}x<0,
\end{equation*}%
which forces $u_{1}^{\ast -}=0$, a contradiction. So assume $u_{2}^{\ast
}<0<u_{1}^{\ast }$. Test the second equation in $(\mathrm{P})$ by $%
-u_{2}^{\ast }{}^{-}$, using $(\mathrm{H}_{5})$ it follows that 
\begin{equation*}
\int_{\Omega }|\nabla u_{2}^{\ast -}|^{p_{2}}\text{ }\mathrm{d}%
x=-\int_{\Omega }f_{2}(x,u_{1}^{\ast },u_{2}^{\ast })u_{2}^{\ast }{}^{-}%
\text{ }\mathrm{d}x<0.
\end{equation*}%
Hence, $u_{2}^{\ast -}=0$, a contradiction. Consequently, $u_{1}^{\ast }$
and $u_{2}^{\ast }$ cannot be of opposite constant sign. However,
considering Theorem \ref{T1} we can conclude that $u_{1}^{\ast }$ and $%
u_{2}^{\ast }$ must change sign simultaneously and therefore, (\ref{55}) is
satisfied. This completes the proof.
\end{proof}

Finally, we indicate an example showing the applicability of Theorems \ref%
{T1}, \ref{T2} and \ref{T4}.

\begin{example}
Consider the functions $f_{1},f_{2}:\Omega \times 
\mathbb{R}
^{2}\rightarrow 
\mathbb{R}
$ defined by 
\begin{equation*}
f_{1}(x,s,t)=(\frac{1}{2}+sgn(t))(|s|^{\alpha _{1}}+|t|^{\beta _{1}})\text{
\ and \ }f_{2}(x,s,t)=(\frac{1}{2}+sgn(s))(|s|^{\alpha _{2}}+|t|^{\beta
_{2}}),
\end{equation*}%
where $sgn(s),sgn(t)$ denote the sign functions, the exponents satisfy $%
-1<\alpha _{1},\beta _{2}<0<\alpha _{2},\beta _{1}$ such that 
\begin{equation*}
\alpha _{2}<\min \left\{ 1,p_{2}-1\right\} ,\text{ \ }\beta _{1}<\min
\left\{ 1,p_{1}-1\right\} 
\end{equation*}%
\begin{equation*}
\text{and \ \ }\alpha _{i}+\beta _{i}>-\min \left\{ 1,p_{i}-1\right\} ,\text{
}i=1,2.
\end{equation*}

It is straightforward to check that assumptions $(\mathrm{H}_{1})-(\mathrm{H}%
_{5})$ are verified. Consequently, Theorems \ref{T1}, \ref{T2} and \ref{T4}
are applicable providing solutions for system $(\mathrm{P})$ with equations
whose right-hand sides are given through the preceding functions $f_{1}$ and 
$f_{2}$. One obtains three nontrivial solutions having sign properties as
described in Theorems \ref{T1} and \ref{T4}. Note that by changing the sign
dependence of the above functions, assuming that $f_{1}$ (resp. $f_{2}$) is
related with $sgn(s)$ (resp. $sgn(t)$), problem $(\mathrm{P})$ is a
decoupled-sign system. In this case, three nontrivial solutions are provided
by Theorems \ref{T1} and \ref{T2}.
\end{example}

\end{document}